\documentclass[12pt]{article}
\usepackage{amsthm}
\usepackage{amscd}
\usepackage{amsmath}
\usepackage{amssymb}
\usepackage{color}
\usepackage{array}
\usepackage{enumerate}
\usepackage[shortlabels]{enumitem}
\setenumerate[1]{itemsep=0pt,partopsep=0pt,parsep=\parskip,topsep=5pt}
\setitemize[1]{itemsep=0pt,partopsep=0pt,parsep=\parskip,topsep=5pt}
\setdescription{itemsep=0pt,partopsep=0pt,parsep=\parskip,topsep=5pt}
\usepackage{graphicx}
\usepackage{hyperref}

\def\cbar{\overline{\CC}}

\def\g{\gamma}
\def\G{\Gamma}

\def\wt{\widetilde}

\def\CC{\mbox{$\mathbb C$}}

\def\ZZ{\mbox{$\mathbb Z$}}

\newtheorem{theorem}{Theorem}[section]
\newtheorem{lemma}[theorem]{Lemma}

\newtheorem{corollary}[theorem]{Corollary}
\newtheorem{definition}{Definition}

\def\smm{{\setminus}}

\def\st{\text{Stab}}
\def\mo{\mathrm{Mon}}
\def\tei{\mathrm{Teich}}

\newtheorem*{theoremA}{Theorem A}

\begin{document}
\title{Rational maps with constant \\ Thurston pullback mapping}
\author{Guizhen Cui \thanks{The first author is supported by NSFC grants [12131016,12071303] and the National Key R\&D Program of China [2021YFA1003203].} and Yiran Wang\footnote{The second author is supported by NSFC grant [12071303].}}
\date{\today}
\maketitle

\begin{abstract}
In this paper, we study CTP maps, that is, marked rational maps with constant Thurston pullback mapping. We prove that all the regular or mixing CTP polynomials satisfy McMullen's condition. Additionally, we construct a new class of examples of CTP maps.
\end{abstract}

\section{Introduction}
In 1980s, Thurston established a topological characterization for post-critically finite branched coverings \cite{DH}. In the proof, he introduced a map on the Teichm{\"u}ller spaces of punctured spheres induced by the post-critically finite branched covering, so called {\it Thurston pullback mapping} following the literature. He showed that the Thurston pullback mapping has an attracting fixed point if and only if for any stable multicurve, the leading eigenvalue of the transition matrix is less than one. Further studies about the Thurston pullback mapping and stable multicurves are still active recently, see for instance \cite{BEK,BEKP,K,KPS,P,S} for more results.

Thurston's theorem was extended to hyperbolic rational maps \cite{CT}. In the case that the Julia set is disconnected, there exists a canonical stable multicurve which characterize the components of the Julia set \cite{CT,CP}. From this point of view, one may expect to derive dynamical invariants from stable multicurves, which are combinatorial invariants. It has been proved that a completely stable multicurve of a post-critically finite rational map induces a combinatorial decomposition of the Julia sets, which leads to a renormalization in certain conditions \cite{CPT,CYY}. As a natural problem, we expect to understand post-critically finite rational maps without completely stable multicurves.

It turns out that a post-critically finite rational map has no completely stable multicurve if the rational map with the post-critically marked set is a CTP map \cite{BEKP}. More precisely, they proved that the Thurston pullback mapping is a constant if and only if each component of the pre-image of a curve in the marked sphere is non-essential. They also provide examples of such rational maps due to McMullen. Later more examples of CTP maps were constructed \cite{BF,Sae1,Sae2}.

In this paper, we characterize CTP polynomials with non-empty regular marked points. Additionally, we construct a new class of examples of CTP maps.

This paper is organized as follows. In \S2, we recall the definitions and some known results about CTP maps, then state our main result. In \S3, we provide some basic properties and tools. A new type of examples is given in \S4. The main result is proved in the last two sections.

\section{Main result}
By a {\bf marked branched covering} $f: (S^2, A)\to (S^2, B)$ we mean a branched covering $f$ of the $2$-sphere $S^2$ with $\deg f\ge 2$, and two finite sets $A, B\subset S^2$ such that $\#A\ge 3$, $\#B\ge 3$ and $f(A)\cup V_f\subset B$, where $V_f$ is the set of critical values of $f$.

Two marked branched coverings $f: (S^2, A)\to (S^2, B)$ and $g: (S^2, A')\to (S^2, B')$ are called {\bf equivalent} if there exists a pair of orientation preserving homeomorphisms $(\phi_0, \phi_1)$ of $S^2$, such that $\phi_0(B)=B'$, $\phi_1(A)=A'$ and the following diagram commutes:
$$
\begin{array}{ccc}
(S^2, A) & \stackrel{\phi_1}{\longrightarrow} & (S^2, A')  \\
f\big\downarrow & & \big\downarrow g \\
(S^2, B) & \stackrel{\phi_0}{\longrightarrow} & (S^2, B').
\end{array}
$$

Denote by $\cbar$ the Riemann sphere. Let $A\subset S^2$ be a finite set with $\#A\ge 3$. Recall that two orientation preserving homeomorphisms $\phi_1, \phi_2: S^2\to\cbar$ are called {\bf Teichm\"{u}ller equivalent rel $A$} if $\phi_2\circ\phi_1^{-1}$ is isotopic to a conformal map of $\cbar$ rel $\phi_1(A)$. The {\bf Teichm\"{u}ller space} $\tei(S^2, A)$ is the space of Teichm\"{u}ller equivalent classes of orientation preserving homeomorphisms $\phi: S^2\to\cbar$. We denote by $[\phi]\in\tei(S^2, A)$ to represent its isotopy class.

Let $f:(S^2, A)\to (S^2, B)$ be a marked branched covering. For any $[\phi_0]\in\tei(S^2, B)$, by the Uniformalization Theorem, there exists a homeomorphism $\phi_1: S^2\to\cbar$ and a rational map $g$ such that the following diagram commutes:
$$
\begin{array}{ccc}
(S^2, A) & \stackrel{\phi_1}{\longrightarrow} & (\cbar, \phi_1(A))  \\
f\big\downarrow & & \big\downarrow g \\
(S^2, B) & \stackrel{\phi_0}{\longrightarrow} & (\cbar, \phi_0(B)).
\end{array}
$$
Moreover, $\phi_1$ is unique up to conformal maps of $\cbar$, and as $\phi_0$ varies in its equivalence class rel $B$, $\phi_1$ stays in the same equivalence class rel $f^{-1}(B)$. In particular, $\phi_1$ stays in its equivalence class rel $A$. Define
$$
\sigma_{f,A,B}:\tei(S^2, B)\to\tei(S^2, A)
$$
by $\sigma_{f,A,B}: [\phi_0]\mapsto[\phi_1]$. It is well-defined and is called the {\bf Thurston pullback mapping} induced by the marked branched covering $f:(S^2, A)\to(S^2, B)$.

\vskip 0.24cm
By definition, $f:(S^2, A)\to (S^2, B)$ is equivalent to $g:(\cbar,\phi_1(A))\to(\cbar,\phi_0(B))$. For any $[\psi_0]\in\tei(\cbar,\phi_0(B))$, there exists a homeomorphism $\psi_1$ of $\cbar$ and a rational map $g_1$ such that the following diagram commutes:
$$
\begin{array}{ccccc}
(S^2, A) & \stackrel{\phi_1}{\longrightarrow} & (\cbar,\phi_1(A)) & \stackrel{\psi_1}{\longrightarrow} & (\cbar, \psi_1\circ\phi_1(A)) \\
f\big\downarrow & & \big\downarrow g & & \big\downarrow g_1 \\
(S^2, B) & \stackrel{\phi_0}{\longrightarrow} & (\cbar,\phi_0(B)) & \stackrel{\psi_0}{\longrightarrow} & (\cbar, \psi_0\circ\phi_0(B)).
\end{array}
$$
Both $\phi_0$ and $\phi_1$ induce bi-holomorphic maps
\begin{equation*}
\begin{split}
\phi_0^*:\,\tei(\cbar, \phi_0(B))\to\tei(S^2, B),\quad [\psi_0]\mapsto[\psi_0\circ\phi_0],  \\
\phi_1^*:\,\tei(\cbar, \phi_1(A))\to\tei(S^2, A),\quad [\psi_1]\mapsto[\psi_1\circ\phi_1],
\end{split}
\end{equation*}
and the following diagram commutes:
$$
\begin{array}{ccc}
\tei(\cbar, \phi_0(B)) & \stackrel{\phi_0^*}{\longrightarrow} & \tei(S^2, B) \\
\sigma_{g,\phi_1(A),\phi_0(B)}\big\downarrow & & \big\downarrow\sigma_{f,A,B} \\
\tei(\cbar, \phi_1(A)) & \stackrel{\phi_1^*}{\longrightarrow} & \tei(S^2, A).
\end{array}
$$
Thus to study the mapping property of Thurston pullback mappings, we only need to consider marked rational maps.

\begin{definition}
A marked rational map with constant Thurston pullback mapping is called a {\bf CTP} map.
\end{definition}

\noindent
{\bf Remark}.
The Thurston pullback mapping is defined originally in the case $A=B$. We adopt the above general definition since we only study the constant case (see \cite{BEKP}).

\vskip 0.24cm
Here are some examples of CTP maps.
\begin{itemize}
\item[(1)] A marked rational map $f:(\cbar, A)\to(\cbar, B)$ is called {\bf trivial} if $\#A=3$. Obviously, a trivial marked rational map is a CTP map.
\item[(2)] A rational map $f$ with $\deg f\ge 2$ is called a {\bf Belyi map} if $\#V_f\le 3$. Let $f:(\cbar, A)\to(\cbar, B)$ be a marked rational map such that $f$ is a Belyi map and $\#(f(A)\cup V_f)=3$. Then it is a CTP map. In fact, for any $[\phi_0]\in\tei(\cbar,B)$, there exists a homeomorphism $\phi_1$ of $\cbar$ and a rational map $g$ such that $g\circ\phi_1=\phi_0\circ f$. Since $\#(f(A)\cup V_f)=3$, $\phi_0$ is isotopic to a conformal map of $\cbar$ rel $f(A)\cup V_f$. Thus $\phi_1$ is isotopic to a conformal map of $\cbar$ rel $A\cup f^{-1}(V_f)$. So $\sigma_{f,A,B}$ is a constant.
\item[(3)]{\bf McMullen's example}. Let $s$ be a Belyi map and $g$ be an arbitrary rational map. If $\#(s(A)\cup V_s)=3$, then $g\circ s:\, (\cbar, A)\to(\cbar, B)$ is a CTP map for any admissible choice of $B$ (refer to \cite{BEKP}).
\end{itemize}

\vskip 0.24cm
A marked rational map $f:(\cbar, A)\to(\cbar, B)$ is called satisfying {\bf McMullen's condition} if there exist a Belyi map $s$ and a rational map $g$ such that $f=g\circ s$ and $\#(s(A)\cup V_s)=3$.

\vskip 0.24cm
Let $f:(\cbar, A)\to(\cbar, B)$ be a marked rational map. Its {\bf regular set} is
$$
E=A\smm f^{-1}(V_f).
$$
The marked rational map will be called {\bf branched} if $E=\emptyset$; {\bf regular} if $E=A$; or {\bf mixing} otherwise. In this paper, we prove the next result.

\begin{theorem}\label{main}
Let $f:(\cbar, A)\to(\cbar, B)$ be a non-trivial regular or mixing CTP polynomial. Then it satisfies McMullen's condition.
\end{theorem}

In addition, we construct a new class of examples of CTP maps, refer to \S4.2.

\section{Basic properties}
\subsection{Marked set}
Let $f:(\cbar, A)\to(\cbar, B)$ be a marked rational map. By definition, it is a CTP map if and only if $f:(\cbar, A)\to(\cbar,f(A)\cup V_f)$ is a CTP map.
For simplicity, in this paper, by $(f,A)$ we mean a marked rational map $f:(\cbar, A)\to(\cbar, B)$ with $B=f(A)\cup V_f$.

\begin{lemma}\label{cup}
Let $(f,A_0)$ and $(f,A_1)$ be CTP maps. If $\#(A_0\cap A_1)\ge 3$, then $(f,A_0\cup A_1)$ is also a CTP map.
\end{lemma}

\begin{proof}
By the Uniformalization Theorem, for any orientation preserving homeomorphism $\phi_0$ of $\cbar$, there is a homeomorphism $\phi_1$ of $\cbar$ and a rational map $g$ such that $g\circ\phi_1=\phi_0\circ f$. Since $(f,A_0)$ is a CTP map, we may require that $\phi_1$ is isotopic to the identity rel $A_0$. Since $(f,A_1)$ is also a CTP map and $\#(A_0\cap A_1)\ge 3$, $\phi_1$ is isotopic to the identity rel $A_1$. This is because any conformal map of $\cbar$ with three fixed points must be the identity. Thus $\phi_1$ is isotopic to the identity rel $A_0\cup A_1$. By definition, $(f,A_0\cup A_1)$ is a CTP map.
\end{proof}

\subsection{Topological characterization}
The next theorem gives a topological characterization of CTP maps \cite{BEKP}. Let $A\subset\cbar$ be a finite set. A Jordan curve $\G\subset\cbar\smm A$ is {\bf essential} in $\cbar\smm A$ if each component of $\cbar\smm\G$ contains at least two points of $A$.

\begin{theoremA}
A marked rational map $(f,A)$ is a CTP map if and only if for any Jordan curve $\G\subset\cbar\smm(f(A)\cup V_f)$, each component of $f^{-1}(\G)$ is non-essential in $\cbar\smm A$.
\end{theoremA}

The following lemma will be used frequently in this paper.

\begin{lemma}\label{top1}
Let $f$ be a rational map with $\deg f\ge 2$. Let $a_0, a_1\in\cbar$ be two points such that $f(a_0)\notin V_f$ and $f(a_0)\neq f(a_1)$. Then there exists a closed arc $\beta: [0,1]\to\cbar$ joining $f(a_0)$ and $f(a_1)$ such that $\beta(0,1)\subset\cbar\smm V_f$ and $f^{-1}(\beta)$ has a component containing $a_0$ to $a_1$.
\end{lemma}

\begin{proof}
Let $\alpha:\,[0,1]\to\cbar$ be an arc joining the points $a_0$ and $a_1$ such that $\alpha(0,1)\subset\cbar\smm f^{-1}(V_f)$. Since $\#V_f<\infty$, there exists a curve $\beta:\,[0,1]\to\cbar$ joining $f(a_0)$ and $f(a_1)$, such that $\beta(0,1)\subset\cbar\smm V_f$, $\beta$ is homotopic to $f(\alpha)$ in $\cbar\smm V_f$ rel the endpoints and the number $k(\beta)$ of self-intersection points of $\beta$ is minimal. We claim that $k(\beta)=0$.

Otherwise, let $t_0\in(0,1)$ be the smallest parameter such that $\beta(t_0)$ is a self-intersectiont point, i.e., there exists $t_1\in (0,1)$ with $t_1\neq t_0$ such that $\beta(t_0)=\beta(t_1)$. Since $f(a_0)\notin V_f$, replacing $\beta(t)$ in a neighborhood of $t_1\in(0,1)$ by a new arc, we obtain a new curve in the homotopy class of $\beta$, such that the number of self-intersection points decreases (see Figure 1). This is a contradiction.

\begin{figure}[htbp]
\centering
\includegraphics[width=6cm]{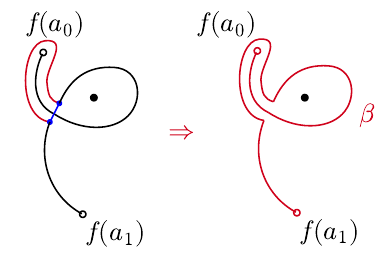}
\caption{From a curve to an arc.}
\end{figure}

Since $\beta$ is homotopic to $f(\alpha)$ in $\cbar\smm V_f$ rel the endpoints, $f^{-1}(\beta)$ has a component containing $a_0$ and $a_1$.
\end{proof}

\begin{corollary}\label{top2}
Let $f$ be a rational map with $\deg f\ge 2$. Let $a_0, a_1\in\cbar$ be two points such that $f(a_0)\notin V_f$ and $f(a_0)\neq f(a_1)$. Then there exists a Jordan domain $\Omega\subset\cbar$ such that
\begin{itemize}
\item $f(a_0), f(a_1)\in\Omega$ and $(\Omega\smm\{f(a_1)\})\cap V_f=\emptyset$,
\item $\partial\Omega$ is disjoint from $V_f$,
\item $f^{-1}(\Omega)$ has a component containing $a_0$ and $a_1$.
\end{itemize}
\end{corollary}

\begin{lemma}\label{marked1}
Let $(f,A)$ be a non-trivial CTP map with regular set $E\neq\emptyset$. Then $f(E)$ is a single point and $\#(A\smm E)\le 2$.
\end{lemma}

\begin{proof}
Assume by contradiction that there are points $a_0, a_1\in E$ such that $f(a_0)\neq f(a_1)$. By Corollary \ref{top2}, there exists a Jordan domain $\Omega\subset\cbar\smm V_f$ such that $\partial\Omega$ is disjoint from $V_f$, both $f(a_0)$ and $f(a_1)$ are contained in $\Omega$ and $f^{-1}(\Omega)$ has a component $U$ containing $a_0$ and $a_1$. Thus $\partial U$ is an essential curve in $\cbar\smm A$. By Theorem A, this is impossible since $(f,A)$ is a CTP map. Thus $\#f(E)=1$.

Assume by contradiction that $\#(A\smm E)>2$. Pick points $a\in E$ and $c\in A\smm E$. By Corollary \ref{top2}, there exists a Jordan domain $\Omega\subset\cbar$ such that $\partial\Omega$ is disjoint from $V_f$, both $f(a)$ and $f(c)$ are contained in $\Omega$, $\Omega\cap V_f=\{f(c)\}$ and $f^{-1}(\Omega)$ has a component $U$ containing $a$ and $c$. In particular, $\#(U\cap(A\smm E))=1$. Since $\#(A\smm E)>2$, $\partial U$ is essential in $\cbar\smm A$. This is a contradiction by Theorem A.
\end{proof}

\subsection{Monodromy group}
Let $(f,A)$ be a CTP map with regular set $E\neq\emptyset$. Then $b=f(E)$ is a single point by Lemma \ref{marked1}. Let $b'$ be a point in $\cbar\smm V_f$ and
$\g:\,[0,1]\to\cbar\smm V_f$ be a curve with $\g(0)=b$ and $\g(1)=b'$. For each point $a\in f^{-1}(b)$, $\g$ has a unique lift $\delta$ such that $f(\delta(t))=\g(t)$ and $\delta(0)=a$. Define $\tau_{\g}(a)=\delta(1)$. Then $\tau_{\g}$ is a bijection from $f^{-1}(b)$ to $f^{-1}(b')$. Denote
$$
A_{\g}=\tau_{\g}(E)\cup(A\smm E).
$$

\begin{lemma}\label{move}
The marked rational map $(f,A_{\g})$ and $(f,A)$ are equivalent.
\end{lemma}

\begin{proof}
There is an isotopy $\phi_t$ ($0\le t\le 1$) of $\cbar$ rel $V_f$ such that $\phi_0=\mathrm{id}$ and $\phi_t(b)=\g(t)$. Thus there is an isotopy $\psi_{t}$ of $\cbar$ rel $f^{-1}(V_f)$ such that $\psi_0=\mathrm{id}$ and $f\circ\psi_t=\phi_t\circ f$ for $t\in[0,1]$. Denote $E_t=\psi_t(E)$ and $A_t=E_t\cup(A\smm E)$. Then the marked rational map $(f,A_t)$ is equivalent to the marked rational map $(f,A)$.
\end{proof}

Now we take $b'=b$. Then $\g$ represents an element of the fundamental group $\pi_1(\cbar\smm V_f, b)$. The above correspondence $\g\to\tau_{\g}$ actually defines a group homomorphism $\tau$ from $\pi_1(\cbar\smm V_f, b)$ to a permutation group on $f^{-1}(b)$. Denote
$$
\mo(f,b)(\text{ or }\mo(f))=\tau(\pi_1(\cbar\smm V_f, b)).
$$
It is called the {\bf monodromy group} of $f$. The next result is a consequence of Lemma \ref{move}.

\begin{corollary}\label{mon}
For any $\tau\in\mo(f,b)$, $(f,\tau(E)\cup(A\smm E))$ is also a CTP map.
\end{corollary}

\begin{lemma}\label{marked2}
Let $(f,A)$ be a CTP map with regular set $E\neq\emptyset$. If $\#(A\smm E)=2$, then $(f,A)$ satisfies McMullen's condition.
\end{lemma}

\begin{proof}
At first, we may assume that $(f,A)$ is a maximal CTP map, i.e., for any finite set $A'\subset\cbar$ with $A'\supset A$, either $A'=A$ or $(f,A')$ is not a CTP map. In fact, let $A'$ be the union of all the subsets $A''$ of $f^{-1}(f(E)\cup V_f)$ such that $A''\supset A$ and $(f,A'')$ are CTP maps. Then $(f,A')$ is a maximal CTP map by Lemma \ref{cup} and \ref{marked1}.

Up to equivalence, we may assume that $A\smm E=\{0,\infty\}$. Pick a point $a\in E$, by Corollay \ref{top2}, there exists a Jordan domain $\Omega\subset\cbar$ such that $\partial\Omega$ is disjoint from $V_f$, both $f(a)$ and $f(0)$ are contained in $\Omega$, $\Omega\cap V_f=\{f(0)\}$ and $f^{-1}(\Omega)$ has a component $U$ containing $a$ and $0$. By Theorem A, $\partial U$ is non-essential. Thus $\#(A\smm U)\le 1$. Since $\infty\notin U$, we have $E\subset U$. Note that $\#E\ge 2$. Thus $d_0:=\deg_{z=0}f\ge 2$.

Let $\rho\in\mbox{Mon}(f,f(a))$ be a permutation induced by a simple closed curve $\g$ in $\Omega$ such that the bounded component of $\cbar\smm\g$ contains $f(0)$. Then for any other else point $a'\in E$, there is an integer $0<n(a')<d_0$ such that $\rho^{n(a')}(a')=a$. Set
$$
n=\inf\{n(a'):\,a'\in E\smm\{a\}\}.
$$
Since $\#(\rho^{n}(E)\cap E)+\#(A\smm E)\ge 3$, $(f,\rho^{n}(E)\cup E\cup(A\smm E))$ is also a CTP map by Lemma \ref{cup} and Corollary \ref{mon}. Thus $\rho^{n}(E)\subset E$ since $(f,A)$ is a maximal CTP map. It follows that $\rho^{n}(E)=E$, $d:=d_0/n$ is an integer and
$$
E=\{a,\rho^{n}(a),\cdots,\rho^{(d-1)n}(a)\}.
$$

As $f(a)$ varies in $\Omega\smm\{f(0)\}$, the cross-ratio
$$
(0,\infty, a, \rho^{n}(a))=\frac{\rho^{n}(a)}{a}
$$
is a constant. By considering the asymptotic behavior as $f(a)\to f(0)$, we obtain
$$
\frac{\rho^{n}(a)}{a}=\zeta:=e^{\frac{2\pi i}{d}}.
$$
Thus $f(\zeta z)=f(z)$ for $z\in U$.

Set $P(z)=z^d$. Then $P(E)$ is a single point and $g(w):=f(P^{-1}(w))$ is a well-defined rational map. Therefore $(f,A)$ satisfies McMullen's condition.
\end{proof}

\section{Examples}
\subsection{Regular case}
The next CTP map found by Saenz does not satisfy McMullen's condition \cite{Sae1}. Here we provide a proof for self-containment. Let
$$
S(z)=\frac{z^{3}(2-z)}{2z-1}.
$$
Then $V_{S}=\{0,1,\infty\}$ and $\deg_{c}S=3$ for $c=0,1,\infty$. Set $A=S^{-1}(b)$ for $b\in\cbar\smm V_{S}$.

\begin{theorem}\label{Saenz}
$(S,A)$ is a CTP map which does not satisfy McMullen's condition.
\end{theorem}

\begin{proof}
Denote $B=\{0,1,b,\infty\}$. For any Jordan curve $\G$ in $\cbar\smm B$, if $\G$ is non-essential in $\cbar\smm B$, then each component of $S^{-1}(\G)$ is non-essential in $\cbar\smm A$.

Now we assume that $\G$ is essential in $\cbar\smm B$. Then one component $U$ of $\cbar\smm\G$ contains the point $b$ and a critical value. Thus $S^{-1}(U)$ has exactly two components. Both of them are Jordan domains, one contains three points in $A$, and another contains one points in $A$. Consequently, each component of $S^{-1}(\G)$ is non-essential in $\cbar\smm A$. By Theorem A, $(S,A)$ is a CTP map. Since $\deg_{c}S=3$ for $c=0,1,\infty$, $S$ does not satisfy McMullen's condition.
\end{proof}

\vskip 0.24cm
Following McMullen's idea, one may produce more CTP maps as the following.

\begin{lemma}\label{Mc}
Let $(f,A)$ be a CTP map. Then for any rational map $g$, $(g\circ f,A)$ is also a CTP map.
\label{com}
\end{lemma}

\begin{proof}
Set $B=f(A)\cup V_{f}$ and $B'=g(B)\cup V_{g}$. Then $B'=g\circ f(A)\cup V_{g\circ f}$. For any $[\phi_0]\in\tei(\cbar,B)$, there are homeomorphisms $\phi_1,\phi_2$ of $\cbar$ and rational maps $f_0,g_0$ such that the following diagram commutes
$$
\begin{array}{ccc}
(\cbar, A) & \stackrel{\phi_2}{\longrightarrow} & (\cbar, \phi_2(A)) \\
f\big\downarrow & & \big\downarrow f_0 \\
(\cbar, B) & \stackrel{\phi_1}{\longrightarrow} & (\cbar, \phi_1(B)) \\
g\big\downarrow & & \big\downarrow g_0 \\
(\cbar, B') & \stackrel{\phi_0}{\longrightarrow} & (\cbar, \phi_0(B')).
\end{array}
$$
Thus
$$
\sigma_{g\circ f,A,B'}=\sigma_{f,A,B}\circ\sigma_{g,B,B'}:\,\tei(\cbar,B')\to\tei(\cbar,B)\to\tei(\cbar,A).
$$
So $\sigma_{g\circ f,A,B'}$ is a constant since $\sigma_{f,A,B}$ is a constant.
\end{proof}

By Lemma \ref{Mc}, for any rational map $g$, $(g\circ S,A)$ is a CTP map. We will call a marked rational map of this form satisfying {\bf Saenz's condition}.

\subsection{Mixing case}
Let
$$
R(z)=-\left[\frac{(z^2-1)(z^2+3)}{4z^2}\right]^3.
$$
It can be factored as $R=P_3\circ g\circ P_2$, where $P_d(z)=z^d$ and $g(z)=-(z-1)(z+3)/(4z)$. The map $g$ has two critical points $(-i\sqrt 3, i\sqrt 3)$, which map to $(e^{\frac{2\pi i}{3}},e^{\frac{4\pi i}{3}})$ by $g$.

\begin{figure}[htbp]
\centering
\includegraphics[width=10cm]{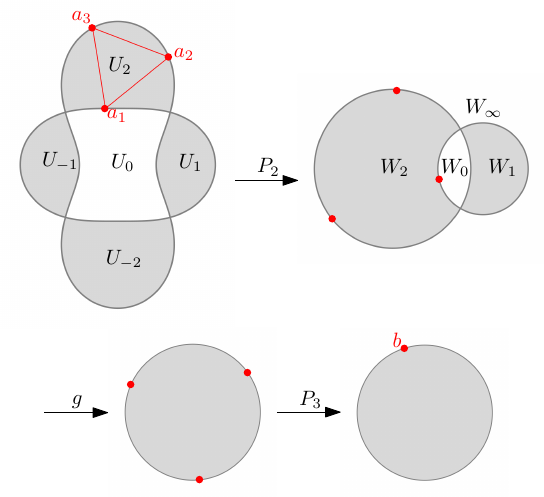}
\caption{The picture of $R^{-1}(S^1)$.}
\end{figure}

It is easy to check that:
\begin{itemize}
\item $V_{R}=\{0,1, \infty\}$.
\item $R^{-1}(\infty)=\{0,\infty\}$, $\deg_{z=0}R=\deg_{z=\infty}R=6$.
\item $R^{-1}(0)=\{\pm 1, \pm i\sqrt{3}\}$, $\deg_{z=c}R=3$ for $c=\pm 1, \pm i\sqrt{3}$.
\item $R^{-1}(1)$ contains $4$ simple critical points which are the roots of the equation $z^4=-3$.
\end{itemize}

Denote by $S^1$ the unit circle. Then $S^1=P_3^{-1}(S^1)$ contains the critical values of $g$. Thus $g^{-1}(S^1)$ divides $\cbar$ into $4$ Jordan domains, $W_0, W_{\infty}, W_1$ and $W_{2}$ which contain the points $\{0,\infty,1,-3\}$, respectively.

Denote by $U_0=P_2^{-1}(W_0)$ and $U_{\infty}=P_2^{-1}(W_{\infty})$. Then $R$ maps these two domains to the outside of the unit disk with degree $6$.
Denote by $U_1, U_{-1}$ be the two components of $P_2^{-1}(W_1)$, and $U_{2}, U_{-2}$ the two components of $P_2^{-1}(W_2)$.
Then $R$ maps these $4$ domains to the unit disk with degree $3$ (see Figure 2).

Pick a point $b\in S^1$ with $b\neq 1$. Then $E=\partial U_{2}\cap R^{-1}(b)$ contains $3$ points, two of them lie on $\partial U_{\infty}$ and one lies on $\partial U_0$. Set $A=E\cup\{\infty\}$.

\begin{theorem}\label{mixing}
$(R,A)$ is a CTP map.
\end{theorem}

\begin{proof}
Denote $B=\{0,1,\infty,b\}$. We want to prove that for any essential Jordan curve $\G\subset\cbar\smm B$, each component of $R^{-1}(\G)$ is non-essential, or equivalently, for any closed arc $\beta$ joining $b$ and a critical value $0,1$ or $\infty$, each component of $R^{-1}(\beta)$ contains either at least three points of $A$ or at most one point of $A$.

Let $\beta_0$ be the line segment $[0,b]$. Then $R^{-1}(\beta_0)$ has $4$ components, one contains three points of $A$ and the others are disjoint from $A$.

Let $\beta_1\subset S^1$ be a closed arc joining the points $1$ and $b$. Then $R^{-1}(\beta_1)$ has $8$ components, all of them contains at most one point of $A$.

Let $\beta_{\infty}=\{bt, t\in [1,\infty]\}$. Then $R^{-1}(\beta_{\infty})$ has two components, one contains three points of $A$, another contains exactly one point of $A$.

\begin{figure}[htbp]
\centering
\includegraphics[width=12cm]{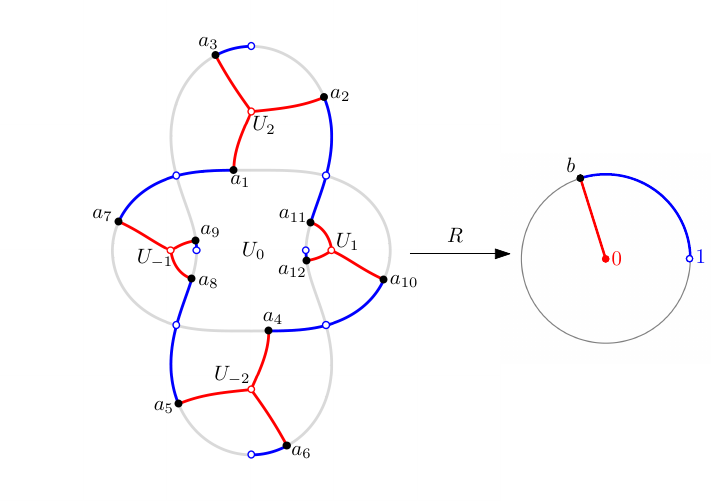}
\caption{Location of $E_k$.}
\end{figure}

For any closed arc $\beta$ joining $b$ and a critical values $0,1$ or $\infty$, there is an element $\gamma\in\pi_1(\cbar\smm V_R,b)$ such that $\beta$ is homotopic to $\g\cdot\beta_i$ ($i=0,1,\infty$) in $\cbar\smm V_R$ rel the endpoint $b$. Thus for each component $Y$ of $R^{-1}(\beta)$, let $Y_i$ be the component of $R^{-1}(\beta_i)$ such that $Y_i$ and $Y$ contains a common point in $R^{-1}(V_R)$, then $Y\cap R^{-1}(b)=\tau_{\g}(Y_i\cap R^{-1}(b))$. Therefore,
$$
\#(Y\cap E)=\#(\tau_{\g}(Y_i\cap R^{-1}(b))\cap E)=\#(Y_i\cap\tau_{\g}^{-1}(E)).
$$
We only need to prove that for any $\tau\in\mo(R,b)$ and any component $Y_i$ of $R^{-1}(\beta_i)$ ($i=0,1,\infty$) with $\infty\notin Y_i$, $\#(Y_i\cap\tau(E))\neq 2$, and for the component $Y_{\infty}$ of $R^{-1}(\beta_{\infty})$ with $\infty\in Y_{\infty}$, $\#(Y_{\infty}\cap\tau(E))\neq 1$.

Let $\rho_{\infty}\in\mo(R,b)$ be induced by $\g_{\infty}(t)=\{2be^{2\pi it}-b:\, 0\le t\le 1\}$. Then $\rho_{\infty}^6=\mathrm{id}$. Set $E_0=E$ and $E_k=\rho_{\infty}^k(E)$ for $1\le k\le 5$. Label the points in $R^{-1}(b)$ as in Figure 3. Then
$$
\begin{array}{l}
\rho_{\infty}:\,a_1\mapsto a_{11}\mapsto a_{12}\mapsto a_4\mapsto a_8\mapsto a_9\mapsto a_1, \\
\rho_{\infty}:\,a_2\mapsto a_{3}\mapsto a_{7}\mapsto a_5\mapsto a_6\mapsto a_{10}\mapsto a_2. \\
E_0=\{a_1,a_2,a_3\},E_1=\{a_{11},a_3,a_7\},E_2=\{a_{12},a_{7},a_5\},\\
E_3=\{a_4,a_5,a_6\},E_4=\{a_8,a_6,a_{10}\},E_5=\{a_9,a_{10},a_2\}.
\end{array}
$$

Let $\rho_0\in\mo(R,b)$ be the permutation induced by $\g_0(t)=\{2be^{2\pi it}/3+b/3:\, 0\le t\le 1\}$ . Then $\rho_0^3=\mathrm{id}$ and
$$
\begin{array}{l}
\rho_0(E_0)=E_0,\, \rho_0(E_3)=E_3, \\
\rho_0:\, E_1\to E_7:=\{a_{12},a_1,a_8\}\to E_5\to E_1, \\
\rho_0:\, E_2\to E_4\to E_6:=\{a_{11},a_9,a_4\}\to E_2, \\
\rho_{\infty}(E_6)=E_7,\, \rho_{\infty}(E_7)=E_6.
\end{array}
$$

Since $\mo(R,b)$ is generated by $(\rho_{\infty},\rho_0)$, we obtain
$$
\{\tau(E):\,\tau\in\mo(R,b)\}=\{E_0,E_1,\cdots,E_7\}.
$$
It is easy to check that for each component $Y_i$ of $R^{-1}(\beta_i)$ ($i=0,1,\infty$) with $\infty\notin Y_i$,
$$
\#(Y_i\cap E_k)\neq 2
$$
for $0\le k\le 7$, and for the component $Y_{\infty}$ of $R^{-1}(\beta_{\infty})$ with $\infty\in Y_{\infty}$,
$$
\#(Y_{\infty}\cap E_k)\neq 1.
$$
Now the proof is complete.
\end{proof}

\subsection{Stablizer of marked points}
In order to see whether a CTP map $(f,A)$ satisfies McMullen's condition or Saenz's condition, we consider the stabilizer of the $\mo(f,b)$. It is defined by
$$
\begin{array}{l}
\st(a)=\{\tau\in\mo(f,b):\, \tau(a)=a\}\text{ for }a\in f^{-1}(b), \\
\st(E)=\{\tau\in\mo(f,b):\, \tau(E)=E\}\text{ for }E\subset f^{-1}(b), \\
\st^*(E)=\bigcap_{a\in E}\st(a).
\end{array}
$$
By definition, $\st^*(E)\subset\st(a)\cap\st(E)$ if $a\in E$.

\begin{lemma}\label{stab1}
Let $(f,A)$ be a CTP map with regular set $E\neq\emptyset$. If it satisfies McMullen's condition, then $\st(a)=\st^*(E)$ for any point $a\in E$.
\end{lemma}

\begin{proof}
By definition, there is a Belyi map $s$ and a rational map $g$ such that $f=g\circ s$ and $\#(s(A)\cup V_s)=3$. Since $E\neq\emptyset$ and $s(E)\notin V_s$, we have $\#V_s=2$.

For each point $a\in E$ and any $\g\in\pi_1(\cbar\smm V_f,b)$ with $\tau_{\g}\in\st(a)$, the lift $\delta$ of $f^{-1}(\g)$ starting from $a$ is a closed curve. Thus $s(\delta)$ is also a closed curve. Since $\#V_2=2$, every lift of $s(\delta)$ under $s$ is a closed curve. Thus $\tau_{\g}(a')=a'$ for any $a'\in E$.
\end{proof}

\begin{lemma}\label{stab2}
Let $(f,A)$ be a regular CTP map satisfying Saenz's condition. Then for each point $a\in A$,
\begin{itemize}
\item[(1)] $\st(a)\smm\st^*(A)\neq\emptyset$,
\item[(2)] $\st(a)\subset\st(A)$, and
\item[(3)] for any $\tau\in\st(a)$, $\tau^3\in\st^*(A)$.
\end{itemize}
\end{lemma}

\begin{proof}
By definition, there is a rational map $g$ such that $f=g\circ S$ and $S(A)$ is a single point, where $S$ is the rational map defined in \S4.1.

(1) Let $\alpha:\,[0,1]\to\cbar\smm(V_S\cup g^{-1}(V_g))$ be a simple loop with $\alpha(0)=\alpha(1)=S(A)$ such that both components of $\cbar\smm\alpha$ contain critical values of $S$. Then one lift $\delta$ of $\alpha$ under $S$ has common endpoint $a'\in A$ and the others have distinct endpoints. Let $\delta':\,[0,1]\to\cbar\smm f^{-1}(V_f)$ be a curve with $\alpha'(0)=a$ and $\alpha'(1)=a'$. Let $\tau\in\mo(f)$ be the permutation induced by $(f(\delta'))^{-1}\cdot g(\alpha)\cdot f(\delta')$. Then $\tau(A)=A$ and $\tau$ has a unique fixed point $a\in A$, i.e., $\tau\in\st(a)\smm\st^*(A)$.

(2) For any $\tau\in\st(a)$, there exists $\g\in\pi_1(\cbar\smm V_f, f(A))$ such that $\tau_{\g}=\tau$ and the lift $\delta$ of $\g$ under $f$ staring from $a$ is a closed curve. Thus $S(\delta)$ is a closed curve and hence $\tau\in\st(A)$.

(3) Suppose $\tau\in\st(a)\smm\st^*(A)$. Since $\#A=4$, $\tau$ has either one or two fixed points. In the first case, $\tau^3\in\st^*(A)$. The second case cannot happen. Otherwise, there exists a point $a_1\in A$ with $a_1\neq a$ such that $\tau(a_1)=a_1$ and $\tau$ commutes the other two points $a_2,a_3$ of $A$. Since $(S,A)$ is a CTP map, there is a conformal map $\lambda$ of $\cbar$ such that $\lambda(a)=a$, $\lambda(a_1)=a_1$, $\lambda(a_2)=a_3$ and $\lambda(a_3)=a_2$. This implies that the cross-ratio of $(a,a_1,a_2,a_3)$ is equal to $-1$. This is impossible since the cross-ratio tends to $e^{\frac{\pi i}{3}}$ as $S(A)\to 0$.
\end{proof}

\begin{lemma}\label{stab3}
Let $(R,A)$ be the CTP map defined in \S4.2. Then for each point $a\in E$,
\begin{itemize}
\item[(a)] $\st(a)\smm\st^*(E)\neq\emptyset$,
\item[(b)] $\st^*(E)=\st(a)\cap\st(E)$, and
\item[(c)] for any $\tau\in\st(a)$, $\tau^2\in\st^*(E)$.
\end{itemize}
\end{lemma}

\begin{proof}
(a) Using the notations in the proof of Theorem \ref{mixing}, we have $E=\{a_1,a_2,a_3\}$ and
$$
\begin{array}{rl}
\rho_0\cdot\rho_{\infty}\cdot\rho_0: & (a_3,a_1,a_2)\mapsto(a_{12},a_1,a_{8}),\,a_1\mapsto a_1, \\
\rho_{\infty}^{-1}\cdot\rho_0: & (a_3,a_1,a_2)\mapsto(a_9,a_{10},a_{2}), \,a_2\mapsto a_2,\\
\rho_{\infty}\cdot\rho_0^{-1}: & (a_2,a_3,a_1)\mapsto(a_{11},a_{3},a_{7}),\, a_3\mapsto a_3.
\end{array}
$$
Thus $\st(a)\smm\st(E)\neq\emptyset$ for each point $a\in E$.

(b) Suppose $\tau\in\st(a)\cap\st(E)$. Assume by contradiction that $\tau\notin\st^*(E)$. Then $\tau$ commutes the other two points of $E$. Since $(R,A)$ is a CTP map, there is a conformal map $\lambda$ of $\CC$ such that $\lambda(a)=a$ and $\lambda$ commutes the other two points of $E$. This implies that the cross-ratio of the four points of $A$ is equal to $-1$. This is impossible since the cross-ratio tends to $e^{\frac{\pi i}{3}}$ as $R(E)$ tends to zero. Thus $\tau\in\st^*(E)$.

(c) One may check that each point of $E$ appears exactly in two distinct sets within $\{\tau(E):\,\tau\in\mo(R,b)\}$. Thus for any $\tau\in\st(a)$, either $\tau(E)=E$ or $\tau^2(E)=E$. In the former case $\tau\in\st(a)\cap\st(E)$. Thus $\tau\in\st^*(E)$ by the statement (b). In the latter case $\tau^2(E)=E$. Thus $\tau^2\in\st^*(E)$ by the same reason.
\end{proof}

\begin{corollary}
$(R,A)$ does not satisfy McMullen's condition and Saenz's condition.
\end{corollary}

\section{CTP Belyi polynomials}
In this section, we will prove Theorem \ref{main} for CTP Belyi polynomials.

\subsection{Branched tree and monodromy}
In this paper, by a {\bf tree}, we mean a finite and connected graph in $\CC$ without loops. Here are some definations and notations used later.

Let $T\subset\CC$ be a tree. For each vertex $c\in T$, its {\bf local degree} $\deg_c T$ is the number of components of $T\smm\{c\}$. It is called an {\bf endpoint} of $T$ if $\deg_cT=1$, or a {\bf branched} vertex of $T$ if $\deg_cT\ge 3$. Two distinct edges of $T$ are {\bf adjacent} if they have common endpoint.

For any collection $\{a_1,\cdots,a_n\}$ ($n\ge 2$) of distinct edges or vertices of $T$, denote by $T[a_1,\cdots,a_n]\subset T$ the minimal subtree of $T$ containing them, i.e., for any subtree $T'\subset T$ containing $\{a_1,\cdots,a_n\}$, $T[a_1,\cdots, a_n]\subset T'$. Denote $|T(a_1,\cdots,a_n)|$ the number of edges in $T[a_1,\cdots,a_n]$ and
$$
T(a_1,\cdots,a_n)=T[a_1,\cdots,a_n]\smm\{\text{endpoints of }T[a_1,\cdots,a_n]\}.
$$

\vskip 0.24cm
Let $f$ be a polynomial with $\#V_f=3$. Let $v_0, v_1$ be the two finite critical values of $f$ and $I\subset\CC$ be a closed arc joining $v_0$ and $v_1$. Then $T=f^{-1}(I)$ is a tree with vertex set $f^{-1}(V_f)\smm\{\infty\}$. We call it a {\bf branched tree} of $f$.

The monodromy group $\mo(f)$ acts on the edges of $T$ as the following: Given a point $b\in I\smm V_f$. For any $\tau\in\mo(f)$ and each edges $e$ of $T$, $\tau(e)$ is the edge containing $\tau(e\cap f^{-1}(b))$.

\begin{figure}[htbp]
\centering
\includegraphics[width=6cm]{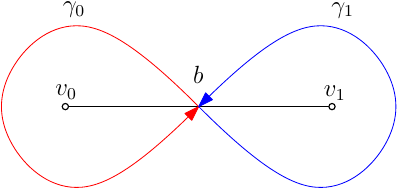}
\caption{The simple closed curve $\g_i$.}
\end{figure}

Let $\rho_i\in\mo(f)$ ($i=0,1$) be the permutation induced by a simple closed curve $\g_i\subset\CC\smm V_f$ such that $\g_i\cap I=\{b\}$ and the bounded component of $\cbar\smm\g_i$ contains exactly one point $v_i$ of $V_f$ (see Figure 4). Then $\mo(f)$ is generated by $\{\rho_0,\rho_1\}$. In particular, if $e,e'$ are edges of $T$ with common endpoint $c\in f^{-1}(v_i)$ and $\rho_i^{k}(e)=e$, then $k/\deg_c f$ is an integer and hence $\rho_i^{k}(e')=e'$.

Define a norm on $\mo(f)$ by $|\tau|=0$ if $\tau=\mathrm{id}$, and
$$
|\tau|=\inf\{n:\,\tau=\rho_{i_n}^{k_n}\cdots \rho_{i_1}^{k_1}\,(k_j\in\ZZ, i_j\in\{0,1\}\}.
$$
Then $|\tau_1\cdot\tau_0|\le|\tau_1|+|\tau_0|$. For each edge $e$ of $T$, $|T(e,\tau(e))|\le|\tau|+1$. Conversely, for any two edges $e,e'$ of $T$, there exists a unique $\tau\in\mo(f)$ such that $\tau(e)=e'$ and $|\tau|=|T(e,e')|-1$.

\begin{lemma}\label{chase}
For any two distinct edges $e$ and $e'$ of $T$, there exists $\tau\in\mo(f)$ such that
\begin{itemize}
\item[(1)] $|T[e,\tau(e)]|=|\tau|+1$,
\item[(2)] $\tau(e)$ and $\tau(e')$ have a common endpoint $c$.
\item[(3)] $T[e,\tau(e)]\cap T[e',\tau(e')]=\{c\}$.
\end{itemize}
\end{lemma}

\begin{proof}
If $e$ and $e'$ are adjacent, then the lemma holds for $\tau=\mathrm{id}$. Now we assume $|T[e,e']|\ge 3$. Let $e_1\subset T[e,e']$ be the edge adjacent to $e$. Then there exists $\tau_1\in\mo(f)$ such that $|\tau_1|=1$ and $\tau_1(e)=e_1$. Denote $e'_1=\tau_{1}(e')$. Then $T(e',e'_1)$ and $e$ are contained in the two distinct components of $T\smm e_1$ (see Figure 5). Either $T[e,e_1]\cap T[e',e'_1]=\emptyset$ or $T[e,e_1]\cap T[e',e'_1]$ is a vertex $c$ of $T$. In the latter case, $c$ is the common endpoint of $e_1$ and $e'_1$. So $\tau_{1}$ satisfies the conditions of the lemma.

\begin{figure}[htbp]
\centering
\includegraphics[width=10cm]{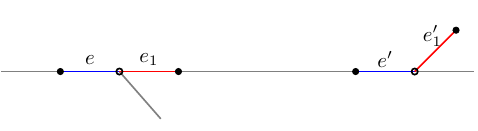}
\caption{Chase of two edges}
\end{figure}

Now we assume  $T[e,e_1]\cap T[e',e'_1]=\emptyset$. Let $e_2\subset T[e_1,e'_1]$ be the edge adjacent to $e_1$. Since $T(e',e'_1)$ and $e$ are contained in distinct components of $T\smm e_1$, we have $|T[e,e_2]|=3$. There exists $\tau_2\in\mo(f)$ such that $|\tau_2|=1$ and $\tau_{2}(e_1)=e_2$. Denote $e'_2=\tau_{2}(e'_1)$. Then $T(e',e_1',e_2')$ and $T(e,e_1)$ are contained in the two distinct components of $T\smm\{e_2\}$. Either $T[e,e_2]\cap T[e',e'_1,e'_2]=\emptyset$ or $T[e,e_2]\cap T[e,e'_1,e'_2]$ is a vertex $c$ of $T$. In the latter case, $c$ is the common endpoint of $e_2$ and $e'_2$. So $\tau_2\cdot\tau_{1}$ satisfies the conditions of the lemma.

Continuing this process successfully, we obtain a sequence $\{\tau_j\}$ in $\mo(f)$ such that $|\tau_j|=1$, $|T[e,e_j]|=j+1$, $T(e',\cdots,e_{j}')$ and $T(e,e_{j-1})$ are contained in the two distinct components of $T\smm\{e_{j}\}$, where $e_j=\tau_j(e_{j-1})$ and $e'_j=\tau_j(e'_{j-1})$. Since $T$ is a finite tree, this process must stop at some step $n\ge 2$. Thus $|T[e,e_n]|=n+1$ and $T[e,e_n]\cap T[e',\cdots,e'_n]$ is a vertex $c$ of $T$. This implies that $c$ is the common endpoint of $e_n$ and $e'_n$. Now the proof is completed.
\end{proof}

\subsection{Location of marked edges}
Let $(f,A)$ be a non-trivial CTP polynomial with regular set $E\neq\emptyset$ and $\#V_f=3$. Let $T$ is a branched tree of $f$ with $E\subset T$. An edge $e$ of $T$ is {\bf marked} if $e\cap E\neq\emptyset$.

\begin{lemma}\label{common}
Suppose that there is a vertex $c\in T$ such that $c$ is the common endpoint of two distinct marked edges of $T$. Then $c$ is the common endpoint of all the marked edges of $T$. Moreover, $T\smm\{c\}$ is disjoint from $(A\smm E)$.
\end{lemma}

\begin{proof}
Let $\{e_1,\cdots,e_n\}$ be all the marked edges of $T$ such that $c$ is the common endpoint of $e_j$ for $1\le j\le n$. By Lemma \ref{mon} and Theorem A, either $A\subset T(e_1,\cdots,e_n)$ or $A\smm T(e_1,\cdots,e_n)$ is a single point $a$. In the former case, the conclusion of the lemma holds. Now we assume the latter case occurs.

{\it Case 1}. $a\in A\smm E$ and hence $a$ is a vertex of $T$.

Let $\tilde e_1\subset T[c,a]$ be the edge with an endpoint $c$. Then there exists $\tau_0\in\mo(f)$ such that $|\tau_0|=1$ and $\tau_0(e_1)=\tilde e_1$. Denote $\tilde e_j=\tau_0(e_j)$ for $1<j\le n$. Then there exists $\tau_1\in\mo(f)$ such that $|\tau_1|=|T[a,\tilde e_1]|-1$ and $a$ is an endpoint of $\tau_1(\tilde e_1)$. Since $|\tau_1|<|T[a,\tilde e_j]|-1$, $a$ is not the endpoint of $\tau_1(\tilde e_j)$ for $2\le j\le n$. Noticing that $\tau_1(\tilde e_1)\cup\{a\}$ contains exactly two points of $\tau_1(\tau_0(E))\cup(A\smm E)$. This is a contradiction by Theorem A.

{\it Case 2}. $a\in E$.

Let $e'$ be the marked edge of $T$ with $a\in e'$. Then $c$ is not an endpoint of $e'$. Let $\tilde e_1\subset T[c,e']$ be the edge with an endpoint $c$. Then there exists $\tau_0\in\mo(f)$ such that $|\tau_0|=1$ and $\tau_0(e_1)=\tilde e_1$. Denote $\tilde e'=\tau_0(e')$ and $\tilde e_j=\tau_0(e_j)$ for $2\le j\le n$ (see Figure 6). Then
$$
|T[\tilde e_1,\tilde e']|<|T[\tilde e_j,\tilde e']|.
$$

\begin{figure}[htbp]
\centering
\includegraphics[width=12cm]{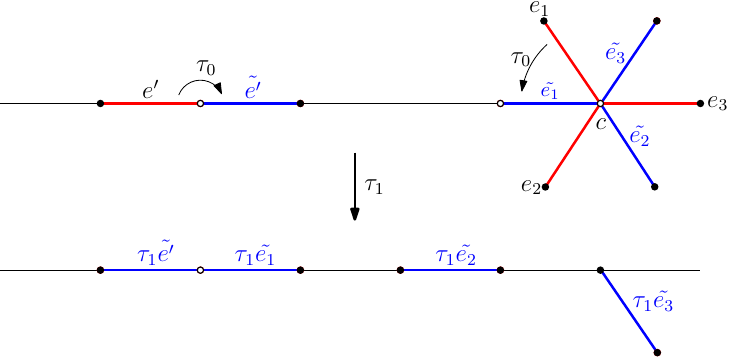}
\caption{Chase of edges}
\end{figure}

Applying Lemma \ref{chase} for $\tilde e_1$ and $\tilde e'$, there exists $\tau_1\in\mo(f)$ such that
\begin{itemize}
\item $|T[\tilde e_1,\tau_1(\tilde e_1)]|=|\tau_1|+1>1$,
\item $\tau_1(\tilde e_1)$ and $\tau_1(\tilde e')$ have a common endpoint $c'$.
\item $T[\tilde e_1,\tau_1(\tilde e_1)]\cap T[\tilde e',\tau_1(\tilde e')]=\{c'\}$.
\end{itemize}
Thus $c'\neq c$ and $\tau_1(\tilde e_j)$ is not adjacent to $\tau_1(\tilde e'))$ for $2\le j\le n$. Since $A\smm T(e_1,\cdots,e_n)=\{a\}$, we have $c'\notin A$. Thus $T(\tau(\tilde e_1),\tau(\tilde e'))$ contains exactly two points of $\tau_1(\tau_0(E))\cup(A\smm E)$. This is a contradiction by Theorem A.
\end{proof}

\begin{lemma}\label{moving}
There is a unique vertex $c_0\in T$ such that $T\smm\{c_0\}$ is disjoint from $A\smm E$ and for any $\tau\in\mo(f)$ and any two marked edges $e,e'$ of $T$,
$$
T[c_0,\tau(e)]\cap T[c_0,\tau(e')]=\{c_0\}\text{ and }|T[c_0,\tau(e)]|=|T[c_0,\tau(e')]|.
$$
\end{lemma}

\begin{proof}
By Lemma \ref{chase} and \ref{common}, there exist a vertex $c_0$ of $T$ and $\tau_0\in\mo(f)$ such that $T\smm\{c_0\}$ is disjoint from $A\smm E$ and $c_0$ is the common endpoint of $\tau_0(e)$ for all the marked edge $e$ of $T$.

{\it Case 1. $\#E=2$}.

Then $c_0\in A$. Let $e_0,e_1$ be the two marked edges of $T$. For any $\tau\in\mo(f,b)$, assume $|T[c_0,\tau(e_0)]|\le |T[c_0,\tau(e_1)]|$. Then there exists $\tau_1\in\mo(f)$ such that $|\tau_1|=|T[c_0,\tau(e_0)]|-1$ and $c_0$ is an endpoint of $\tau_1(\tau(e_0))$.

If $c_0$ is not an endpoint of $\tau_1(\tau(e_1))$, then $T(c_0,\tau_1(\tau(e_0)))$ contains exactly two points of $A$. This is a contradiction by Theorem A. Thus $|T[c_0,\tau(e_0)]|=|T[c_0,\tau(e_1)]|$ and $c_0$ is the common endpoint of $\tau_1(\tau(e_0))$ and $\tau_1(\tau(e_1))$.

If $T[c_0,\tau(e_0)]$ and $T[c_0,\tau(e_1)]$ contains a vertex $c'\neq c_0$ of $T$, then $|T[c',\tau(e_0)]|=|T[c',\tau(e_1)]|$. There exists $\tau_2,\tau'\in\mo(f)$ such that $|\tau_1|=|\tau_2|+|\tau'|$, $|\tau_2|=|T[c',\tau(e_0)]|-1$ and $c'$ is an endpoint of $\tau_2(\tau(e_0))$. This implies that $c'$ is the common endpoint of $\tau_2(\tau(e_0))$ and $\tau_2(\tau(e_1))$. This is a contradiction by Theorem A since $c'\notin A$.

{\it Case 2. $\#E\ge 3$}.

At first, we claim that for any $\tau\in\mo(f)$ and any three distinct marked edges $e_i$ ($i=0,1,2$) of $T$, $T[\tau(e_0),\tau(e_1),\tau(e_2)]$ has a branched vertex. Otherwise, assume $\tau(e_0)\subset T[\tau(e_1),\tau(e_2)]$. By Lemma \ref{chase}, there exists $\tau_1\in\mo(f)$ such that
\begin{itemize}
\item $|T[\tau(e_0),\tau_1(\tau(e_0))]|=|\tau_1|+1$,
\item $\tau_1(\tau(e_0))$ and $\tau_1(\tau(e_1))$ have a common endpoint $c$,
\item $T[\tau(e_0),\tau_1(\tau(e_0))]\cap T[\tau(e_1),\tau_1(\tau(e_1))]=\{c\}$.
\end{itemize}
Thus $\tau_1(\tau(e_2))$ is not adjacent to $\tau_1(\tau(e_1))$. This is a contradiction by Lemma \ref{common}.

Now we claim that for any $\tau\in\mo(f)$, there is a vertex $c(\tau)$ of $T$ such that for any two distinct marked edges $e,e'$ of $T$,
$$
T[c(\tau),\tau(e)]\cap T[c(\tau),\tau(e')]=\{c(\tau)\}\text{ and }|T[c(\tau),\tau(e)]|=|T[c(\tau),\tau(e')]|.
$$

Let $T_{\tau}\subset T$ be the minimal subtree containing all the $\tau$-images of marked edges of $T$. Then $T_{\tau}$ has branched vertex. Let $c'$ be a branched vertex of $T_{\tau}$ and $e_0$ be a marked edge of $T$ such that
$$
|T[c',\tau(e_0)]|\le |T[c,\tau(e)]|
$$
for all the branched points $c$ of $T_{\tau}$ and marked edges $e$ of $T$. Then $T(c',\tau(e_0))$ contains no branched point of $T_{\tau}$. For any marked edge $e$ of $T$, $\tau(e)$ is not contained in $T[c',\tau(e_0)]$ by the first claim. Thus for any marked edge $e$ of $T$,
$$
T[c',\tau(e_0)]\cap T[c',\tau(e)]=\{c'\}\text{ and }|T[c',\tau(e_0)]|\le|T[c',\tau(e)]|
$$
Assume that there is a marked edge $e_1$ such that $|T[c',\tau(e_1)]|>|T[c',\tau(e_0)]|$. Then there exists $\tau_1\in\mo(f)$ such that $|\tau_1|=|T[c',\tau(e_0)]|-1$ and $c'$ is an endpoint of $\tau_1(\tau(e_0))$. Thus for any marked edge $e$ of $T$, $\tau(e)$ and $\tau_1(\tau(e))$ are contained in the same component of $T\smm\{c'\}$. In particular, $c'$ is not an endpoint of $\tau_1(\tau(e_1))$.

\begin{figure}[htbp]
\centering
\includegraphics[width=12cm]{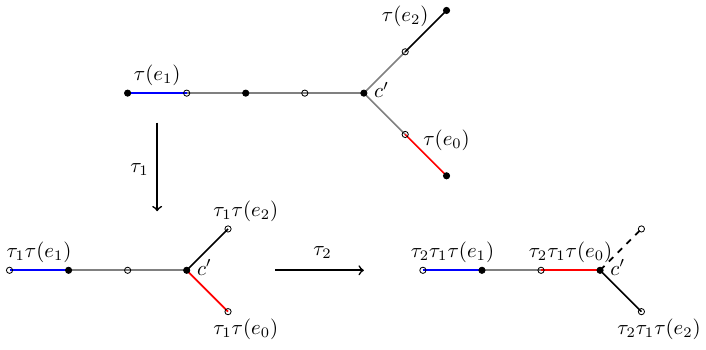}
\caption{Center of marked edges}
\end{figure}

There exists $\tau_2\in\mo(f)$ such that $|\tau_2|=1$ and $\tau_2(\tau_1(\tau(e_1)))$ and $\tau_2(\tau_1(\tau(e_0)))$ are contained in the same component of $T\smm\{c'\}$ (see Figure 7). This is a contradiction by the first claim. Thus we have
$$
|T[c',\tau(e)]|=|T[c',\tau(e_0)]|
$$
for any marked edge $e$ of $T$. Now the second claim is proved.

For any $\tau,\tau'\in\mo(f)$, $c(\tau'\cdot\tau)=c(\tau)$ if $|\tau'|=1$. Thus $c(\tau)=c_0$.
\end{proof}

\begin{lemma}\label{stab}
Let $(f,A)$ be a non-trivial CTP polynomial with regular set $E\neq\emptyset$ and $\#V_f=3$. Then $\st(a)=\st^*(E)$ for any point $a\in E$.
\end{lemma}

\begin{proof}
Let $T$ be a branched tree of $f$ with $E\subset T$. Let $e$ be the edge of $T$ with $a\in e$. For any point $a'\in E$, let $e'$ be the edge of $T$ with $a'\in e'$. By Lemma \ref{moving}, there is a unique vertex $c_0\in T$ such that $T\smm\{c_0\}$ is disjoint from $A\smm E$ and for any $\tau\in\mo(f)$,
$$
T[c_0,\tau(e)]\cap T[c_0,\tau(e')]=\{c_0\}\text{ and }|T[c_0,\tau(e)]|=|T[c_0,\tau(e')]|.
$$
Without loss of generality, we may assume that $f(c_0)=v_0$ and $c_0$ is the common endpoint of $e$ and $e'$.

For any $\tau\in\st(a)$, there exists a sequence $\{\tau_1,\cdots,\tau_n\}$ in $\mo(f)$ such that $|\tau_j|=1$ for $1\le j\le n$ and $\tau=\tau_n\cdot\cdots\cdot\tau_1$. Denote $e_0=e$ and $e_j=\tau_{j}(e_{j-1})$. Then either $e_{j}=e_{j-1}$ or $e_{j}$ is adjacent to $e_{j-1}$.
Let $\{e_{j_k}\}$ ($0\le k\le m$) be the subsequence of $\{e_j\}_{j=0}^n$ such that $c_0$ is an endpoint of $e_{j_k}$. Then $j_0=0$.

Denote $e'_0=e'$ and $e'_j=\tau_j(e'_{j-1})$ for $1\le j\le n$. By Lemma \ref{moving}, $\{e'_{j_k}\}$ ($0\le k\le m$) is also the subsequence of $\{e'_j\}_{j=0}^n$ such that $c_0$ is an endpoint of $e'_{j_k}$.

If $j_{k+1}>j_{k}+1$, then $\{e_j\}$ are contained in the same component of $T\smm\{c_0\}$ for $j_{k}\le j\le j_{k+1}$. Thus $e_{j_{k+1}}=e_{j_k}$. Similarly, $e'_{j_{k+1}}=e'_{j_k}$.

If $j_{k+1}=j_{k}+1$, then $\tau_{j_k}=\rho_0^{p_k}$ for some integer $p_k\in\ZZ$. Thus there is an integer $p\in\ZZ$ such that
$$
e_{j_m}=\rho_0^{p}(e)\text{ and }e'_{j_m}=\rho_0^{p}(e').
$$
Since $\tau\in\st(a)$, we have $e_{j_m}=e$. Thus $\rho_0^{p}(e)=e$ and hence $e'_{n}=\rho_0^{p}(e')=e'$. So $\tau(a')=a'$. Now the lemma is proved.
\end{proof}

\subsection{Power factor}
\begin{lemma}\label{symmetry}
Let $(f,A)$ be a non-trivial CTP polynomial with regular set $E\neq\emptyset$. Suppose that $\st(a)=\st^*(E)$ for any point $a\in E$. Then $(f,A)$ satisfies McMullen's condition.
\end{lemma}

\begin{proof}
Denote $b=f(E)$. Let $(a_i,a_j)$ be a pair of points in $E$. For any point $z\in\CC\smm f^{-1}(V_f)$, let $\alpha_i:\,[0,1]\to\CC\smm f^{-1}(V_f)$ be a curve joining $a_i$ and $z$. Then $\g:=f(\alpha_i)\subset\CC\smm V_f$ is a curve joining $b$ and $f(z)$. Thus there is a unique lift $\alpha_j$ of $f^{-1}(\g)$ joining $a_j$ and a point $w$ with $f(w)=f(z)$. Since $\st(a_i)=\st(a_j)$, $w$ is independent of the choice of the curve $\alpha_i$. Define
$$
\lambda_{i,j}(z)=w.
$$
Then $\lambda_{i,j}$ is a conformal map of $\CC\smm f^{-1}(V_f)$ and $f\circ\lambda_{i,j}=f$ on $\cbar$. In particular, $\lambda_{i,j}$ is a conformal map of $\CC$ since $\lambda_{i,j}(z)\to\infty$ as $z\to\infty$.

Let $\Lambda$ be the group generated by $\{\lambda_{i,j}\}$ for all pairs of points $(a_i,a_j)$ in $E$. Then $f\circ\lambda=f$ for all $\lambda\in\Lambda$. Thus $d:=\#\Lambda\le\deg f$. Therefore $\Lambda$ must be a cyclic group generated by a rotation $\lambda_0(z)=\zeta(z-c)+c$ with $c\in\CC$ and $\zeta=e^{\frac{2\pi i}{d}}$. Set $P(z)=(z-c)^d$. Then $P\circ\lambda_0=P$ and $P(E)$ is a single point. Set $g(w)=f(P^{-1}(w))$. Then $g$ is a well-defined rational map and $f=g\circ P$.

If $\#E=2$, then $\#(A\smm E)=2$. By Lemma \ref{marked2}, $(f,A)$ satisfies McMullen's condition. Assume $\#E\ge 3$. Let $U\subset\CC\smm V_f$ be a Jordan domain with $b\in U$. Then $g^{-1}(U)$ has a unique component $W$ such that $P(E)\in W$. Pick a point $w'\in W$ such that $|w'|>|P(E)|$. Since $(f,A)$ is a CTP map, there exists a conformal map $\lambda$ of $\cbar$ such that $\lambda(A\smm E)=A\smm E$ and $\lambda(E)\subset P^{-1}(w')$. In particular, the round circle containing $E$ maps to the round circle containing $P^{-1}(w')$ by $\lambda$. This implies that $A\smm E\subset\{c,\infty\}$. Thus $\#(P(A)\cup V_P)=3$ and hence $f$ satisfies McMullen's condition.
\end{proof}

\noindent
{\it Proof of Theorem \ref{main} in Belyi case}. Combining Lemma \ref{stab} and \ref{symmetry}, we complete the proof. \qed

\section{CTP polynomials}
In this section we will prove Theorem \ref{main} in general case.

\subsection{Pinching of CTP polynomials}
Let $(f,A)$ be a non-trivial CTP polynomial with regular set $E\neq\emptyset$ and $\#V_f\ge 4$. Denote $b=f(E)$. Let $D\subset\CC\smm\{b\}$ be a Jordan domain such that $\partial D$ is disjoint from $V_f\cup\{b\}$ and $\#(V_f\smm D)=2$. By Theorem A, for each component $U$ of $f^{-1}(D)$, $\#(U\cap A)\le 1$ since $\#E\ge 2$ by Lemma \ref{marked1}.

Pick a point $v\in D$. For each component $U$ of $f^{-1}(D)$, Define $\wt g_U:\, U\to D$ to be a branched covering such that the following conditions hold:
\begin{itemize}
\item $\wt g_U$ has no critical point in $U\smm A$.
\item If $U$ contains a point $a\in A$, then $g_U(a)=v$.
\item $\wt g_U$ can be continuously extended to the boundary with $(\wt g_U)|_{\partial U}=f|_{\partial U}$.
\end{itemize}
Define $\wt g=f$ on $\cbar\smm f^{-1}(D)$ and $\wt g=g_U$ on every component $U$ of $f^{-1}(D)$. Then $\wt g$ is a branched covering of $\cbar$. By the Uniformalization Theorem, there is a homeomorphism $\varphi$ of $\CC$ such that $g=\wt g\circ\varphi^{-1}$ is a polynomial. In particular, $g\circ\varphi=f$ on $f^{-1}(\Omega)$.

Denote $\Omega=\cbar\smm\overline{D}$. Then $V_g=(\Omega\cap V_f)\cup\{v\}$. Thus $g$ is a Belyi polynomial. Denote $A'=\varphi(A)$. Then $g(A')\subset V_g\cup\{b\}$.
We call $(g,A')$ a {\bf pinching} of $(f,A)$ on $D$.

\begin{lemma}\label{pinching}
$(g,A')$ is a non-trivial CTP Belyi polynomial.
\end{lemma}

\begin{proof}
For any Jordan curve $\G\subset\cbar\smm(V_g\cup\{b\})$, since $D$ contains at most one point of $g(A')\cup V_g$, there is a Jordan curve $\G'\subset\Omega$ homotopic to $\G$ rel $g(A')\cup V_g$. By Theorem A, each component of $f^{-1}(\G')$ is non-essential in $\cbar\smm A$. Thus each component of $g^{-1}(\G')=\varphi(f^{-1}(\G'))$ is non-essential in $\cbar\smm A'$. By Theorem A, $(g,A')$ is a non-trivial CTP polynomial.
\end{proof}

\vskip 0.24cm
For any $\g\in\pi_1(\CC\smm V_g,b)$, $\g$ can be chosen such that $\g\subset\Omega$. So $\g$ also represents an element of $\pi_1(\CC\smm V_f,b)$. Let $\tau_{f,\g}\in\mo(f,b)$ and $\tau_{g,\g}\in\mo(g,b)$ be the permutations induced by $\g$, respectively. Then we obtain a group homomorphism
$$
\varphi^*:\,\mo(g,b)\to\mo(f,b),\quad\varphi^*(\tau_{g,\g})=\tau_{f,\g}.
$$
Denote by $\st(\varphi(a))\subset\mo(g,b)$ the stablizer of $\varphi(a)$ for $a\in f^{-1}(b)$.

\begin{lemma}\label{injection}
The group homomorphism $\varphi^*$ is injective. Moreover, for each $a\in f^{-1}(b)$ and $\tau\in\mo(g,b)$, $\tau\in\st(\varphi(a))$ if and only if $\varphi^*(\tau)\in\st(a)$.
\end{lemma}

\begin{proof}
For any $\tau\in\mo(g,b)$, there exists a curve $\g\in\pi_1(\CC\smm V_g,b)$ with $\g\subset\Omega$ such that $\tau_{g,\g}=\tau$. If $\tau\neq\mathrm{id}$, then $\g$ has a lift $\delta$ under $g$ such that the start point and the endpoint of $\delta$ are different. Since $g\circ\varphi=f$ on $f^{-1}(\Omega)$, $\varphi^{-1}(\delta)$ is a lift of $\g$ under $f$ whose start point and endpoint are different. Thus $\tau_{f,\g}\neq\mathrm{id}$. So $\varphi^*$ is injective.

Note that $\tau\in\st(\varphi(a))$ if and only if $\g$ has a lift $\delta$ under $g$ such that both the start point and the endpoint of $\delta$ is the point $\varphi(a)$, or equivalently, both the start point and the endpoint of $\varphi^{-1}(\delta)$ is the point $a$. Thus $\tau\in\st(\varphi(a))$ if and only if $\varphi^*(\tau)\in\st(a)$.
\end{proof}

\subsection{Alternate pinching}
\begin{lemma}\label{stabm}
Let $(f,A)$ be a non-trivial CTP polynomial with regular set $E\neq\emptyset$ and $\#V_f\ge 4$. Then $\st(a)=\st^*(E)$ for any point $a\in E$.
\end{lemma}

\begin{proof}
Denote $b=f(E)$ and $V_f=\{v_1,\cdots,v_n,\infty\}$. Let $\rho_{\infty}\in\mo(f,b)$ be the permutation induced by a simple closed curve $\g_{\infty}\subset\cbar\smm V_f$ such that the bounded component $D$ of $\cbar\smm\g_{\infty}$ contains $V_f\smm\{\infty\}$. Let $\rho_i\in\mo(f,b)$ ($1\le i\le n$) be the permutation induced by a simple closed curve $\g_i\subset D\smm V_f$ such that $\g_i\cap\g_{\infty}=\{b\}$ and the bounded component of $\cbar\smm\g_i$ contains exactly one point $v_i$ of $V_f$ (see Figure 8). Then $\mo(f,b)$ is generated by $\{\rho_i\}$.

For each $1\le i\le n$, there exists a Jordan domain $D_i\subset\cbar\smm\{\g_i\cup\g_{\infty}\}$ such that $V_f\smm\{v_i,\infty\}\subset D_i$. Let $(g_i, A_i)$ be a pinching of $(f,A)$ on $D_i$ and $\varphi^*_i:\,\mo(g_i,b)\to\mo(f,b)$ be the group homomorphism defined in Lemma \ref{injection}. Then both $\rho_{\infty}$ and $\rho_{i}$ are contained in $\varphi^*_i(\mo(g_i,b))$.

\begin{figure}[htbp]
\centering
\includegraphics[width=6cm]{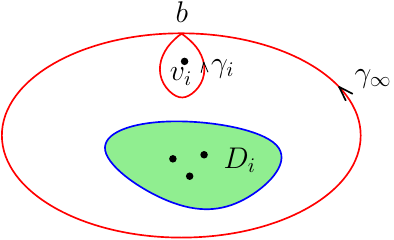}
\caption{Alternate pinching}
\end{figure}

For any $\tau\in\st(a)$, there is a sequence $\{i_1,\cdots,i_m\}$ with $i_j\in\{1,\cdots,n\}$ such that
$$
\tau=\rho_{i_m}\cdot\cdots\cdot\rho_{i_1}.
$$
There is an integer $k_1\in\ZZ$ such that $\tau_1:=\rho_{\infty}^{k_1}\cdot\rho_{i_1}\in\st(a)$. There is also an integer $k_2\in\ZZ$ such that $\tau_2:=\rho_{\infty}^{k_2}\cdot\rho_{i_2}\cdot\rho_{\infty}^{-k_1}\in\st(a)$. Inductively, for each $1<j<m$, There is an integer $k_j\in\ZZ$ such that $\tau_j:=\rho_{\infty}^{k_j}\cdot\rho_{i_j}\cdot\rho_{\infty}^{-k_{j-1}}\in\st(a)$. Set $\tau_m=\rho_{i_m}\cdot\rho_{\infty}^{-k_{m-1}}$.
Then
$$
\tau=\tau_m\cdot\cdots\cdot\tau_1,
$$
and $\tau_j\in\st(a)$ for $1\le j\le m$.

There exists $\wt\tau_j\in\mo(g_{i_j},b)$ ($1\le j\le m$) such that
$$
\begin{array}{l}
\tau_1=\rho_{\infty}^{k_1}\cdot\rho_{i_1}=\phi^*_{i_1}(\wt\tau_1), \\
\tau_j=\rho_{\infty}^{k_j}\cdot\rho_{i_j}\cdot\rho_{\infty}^{-k_{j-1}}=\phi^*_{i_j}(\wt\tau_j)\text{ for }1<j<m, \\
\tau_m=\rho_{i_m}\cdot\rho_{\infty}^{-k_{m-1}}=\phi^*_{i_m}(\wt\tau_m).
\end{array}
$$
By Lemma \ref{injection}, $\wt\tau_j\in\st(\varphi_{i_j}(a))$. For any point $a'\in E$, we have $\wt\tau_j\in\st(\varphi_{i_j}(a'))$ by Lemma \ref{pinching} and \ref{stab}. Again by Lemma \ref{injection}, $\tau_j\in\st(a')$. Thus $\tau(a')=a'$.
\end{proof}

\noindent
{\it Proof of Theorem \ref{main}}. Combining Lemma \ref{stab}, \ref{stabm} and \ref{symmetry}, we complete the proof. \qed

\appendix
\section{An elementary proof of Theorem \ref{mixing}}
The following proof is provided by Jie Cao.

\vskip 0.24cm
\noindent
{\it Proof}.
Denote $b=k^3$ and $\zeta=e^{\frac{2\pi i}{3}}$. Then
\begin{equation*}
\begin{split}
& a_1=i\sqrt{1+2k+2\sqrt{1+k+k^2}}, \\
& a_2=i\sqrt{1+2k\zeta+2\sqrt{1+k\zeta+k^2\zeta^2}}, \\
& a_3=i\sqrt{1+2k\zeta^2+2\sqrt{1+k\zeta^2+k^2\zeta}},
\end{split}
\end{equation*}
We want to prove
$$
\frac{a_1-a_2}{a_3-a_2}=e^{\frac{\pi i}{3}},
$$
or equivalently $a_1+a_2\zeta+a_3\zeta^2=0$. Denote
\begin{equation*}
\begin{split}
& t_1=1+k+k^2, \\
& t_2=1+k\zeta+k^2\zeta^2, \\
& t_3=1+k\zeta^2+k^2\zeta.
\end{split}
\end{equation*}
Then $a_1+a_2\zeta+a_3\zeta^2=0$ is equivalent to
$$
\sqrt{1+2k+2\sqrt{t_1}}+\zeta\sqrt{1+2k\zeta+2\sqrt{t_2}}=-\zeta^2\sqrt{1+2k\zeta^2+2\sqrt{t_3}}.
$$
Take square and applying the equality $1+\zeta+\zeta^2=0$, we obtain
$$
\sqrt{1+2k+2\sqrt{t_1}}\sqrt{1+2k\zeta+2\sqrt{t_2}}=(1-k\zeta^2)-\zeta^2\sqrt{t_1}-\zeta\sqrt{t_2}+\sqrt{t_3}.
$$
Take square again, we obtain
$$
\sqrt{t_1t_2}+\zeta\sqrt{t_2t_3}+\zeta^2\sqrt{t_1t_3}=(1-k)\zeta\sqrt{t_1}+(\zeta^2-k)\sqrt{t_2}+(1-k\zeta^2)\sqrt{t_3}.
$$
Again,
\begin{equation*}
\begin{split}
&[\zeta t_2-(1-k)(\zeta-k)]\sqrt{t_1t_3} \\
=\, & [(1-k)(1-k\zeta)-t_3]\sqrt{t_1t_2}+[(\zeta^2-k)(1-k\zeta^2)-\zeta^2t_1]\sqrt{t_2t_3}.
\end{split}
\end{equation*}
All these three coefficients are equal to $0$.
\qed

\noindent
Guizhen Cui \\
School of Mathematical Sciences, \\
Shenzhen University, \\
Shenzhen, P. R. China \\
gzcui@szu.edu.cn

\vskip 0.24cm
\noindent
Yiran Wang \\
Beijing International Center for Mathematical Research (BICMR) \\
Peking University \\
Beijing, P. R. China \\
wangyiran@amss.ac.cn
\end{document}